\documentclass[preprint,11pt]{amsart}
\usepackage[margin=1in]{geometry}              
\usepackage{amsmath}               
\usepackage{amsfonts}              
\usepackage{amsthm}
\usepackage{enumerate}
\usepackage{esint}
\usepackage{accents}
\usepackage{enumitem}

\numberwithin{equation}{section}

\newtheorem{thm}{Theorem}[section]
\newtheorem{defn}[thm]{Definition}
\newtheorem{lem}[thm]{Lemma}

\newtheorem{cor}[thm]{Corollary}

\def\XXint#1#2#3{{\setbox0=\hbox{$#1{#2#3}{\int}$}
     \vcenter{\hbox{$#2#3$}}\kern-.5\wd0}}

\begin{document}

\nocite{*}

\title[Tug of war]{A game-theoretic approach to the parabolic normalized p-Laplacian obstacle problem  }


\author{Hamid EL Bahja}

\address{Hamid EL Bahja, The African Institute for Mathematical Sciences, Research and Innovation Centre, Rwanda}
\email{hamidsm88@gmail.com, helbahja@aims.ac.za}


\maketitle

\begin{abstract}

This paper establishes a probabilistic representation for the solution of the parabolic obstacle problem associated with the normalized $p$-Laplacian. We introduce a zero-sum stochastic tug-of-war game with noise in a space-time cylinder, where one player has the option to stop the game at any time to collect a payoff given by an obstacle function. We prove that the value functions of this game exist, satisfy a dynamic programming principle, and converge uniformly to the unique viscosity solution of the continuous obstacle problem as the step size $\varepsilon$ tends to zero. \end{abstract}

\section{Introduction}

The theory of obstacle problems represents a cornerstone in the study of free boundary problems, with applications ranging from membrane elasticity to financial mathematics \cite{fri,chap,pha,ham1}. Obstacle problems, where the solution is constrained to remain above or under a prescribed function, have been extensively analyzed in both elliptic and parabolic cases \cite{Caffarelli1977, Kinderlehrer1980,ham2}. 

In parallel developments, probabilistic approaches to nonlinear PDEs have yielded deep insights through connections with stochastic games. A particularly fruitful line of research has established relationships between certain degenerate elliptic operators and tug-of-war games. The pioneering work of \cite{PeresSchramm2009} connected deterministic tug-of-war games to the infinity Laplacian, while \cite{PeresSheffield2008} introduced noise into these games to obtain the normalized $p$-Laplacian:
\[
\Delta_p^N u = |\nabla u|^{2-p} \text{div}(|\nabla u|^{p-2}\nabla u) = \Delta u + (p-2)\Delta_{\infty} u,
\]
where
\[
\Delta_\infty u = \frac{\langle D^2u \, \nabla u, \nabla u \rangle}{|\nabla u|^2}.
\]
See also the works \cite{man1,man2} and references therein for more results in the elliptic case. In the time-dependent setting, game-theoretic formulations and asymptotic mean-value characterizations for a class of nonlinear parabolic equations have been treated by several authors, see for instance~\cite{man3,parv,DelPezzoRossi2014,Han2022}. 
For obstacle problems in the elliptic case, the tug of war game with noise and optimal-stopping framework has been adapted to yield value-function approximations that converge to viscosity solutions of obstacle problems for $p$-Laplace-type operators~\cite{man4,lew}. However, despite these separate advancements, a significant gap remains. The parabolic obstacle problem for the normalized $p$-Laplacian has not been studied specially and systematically from a game-theoretic perspective.

In this paper we bridge that gap. Building directly on the elliptic obstacle framework of \cite{lew} and the parabolic game constructions in \cite{man3}, we introduce a parabolic tug-of-war with noise that incorporates an optimal stopping rule and prove that its value functions converge uniformly to the viscosity solution of the parabolic obstacle problem for the normalized $p$-Laplacian. Our approach combines martingale estimates adapted to stopping times, explicit near-boundary barriers, and viscosity techniques to identify the limit and establish uniqueness.

Our main contributions are threefold. First, we establish that our game has a value and that this value satisfies a dynamic programming principle (DPP), which serves as the discrete analogue of the PDE. Second, we prove uniform Hölder estimates for the value functions, independent of the discretization parameter $\varepsilon$, providing the compactness needed to pass to the continuum limit. Finally, we show that as $\varepsilon \to 0$, these value functions converge uniformly to the unique viscosity solution of the parabolic normalized $p$-Laplacian obstacle problem:
\begin{equation}
\begin{cases}
\min\left\{ (n+p)u_t - \left[(p-2)\Delta_\infty u + \Delta u\right],\ u-\psi \right\} = 0 & \text{in } \Omega_T, \\
u = F & \text{on } \partial_p \Omega_T, \\
u \geq \psi & \text{in } \Omega_T,
\end{cases}
\end{equation}
where $F$ is the boundary data and $\psi$ is the obstacle constraint. \medskip Throughout the paper we assume
\[
2\leq p<\infty.
\]
In particular, with
\[
\alpha=\frac{p-2}{p+n},\qquad \beta=\frac{n+2}{p+n},
\]
we have \(0<\alpha<1\), \(0<\beta<1\) and \(\alpha+\beta=1\); thus \(\alpha,\beta\) may be interpreted as the weights appearing in the discrete DPP and in the probabilistic game.
\medskip

The paper is organized as follows. In Section 2, we introduce the necessary preliminaries and formally define the parabolic tug-of-war game with an obstacle. Section 3 is devoted to the dynamic programming principle and the existence of the game value. Section 4 contains the convergence analysis, including uniform regularity estimates and the identification of the limit as the unique viscosity solution of the continuous obstacle problem.

\section{Preliminaries}
Throughout the paper, $\Omega\in \mathbb{R}^{N}$ is a bounded domain. For $T>0$, let $\Omega_{T}=\Omega\times(0,T)$ be a parabolic cylinder with the parabolic boundary
\begin{equation*}
    \Gamma_p=\{\partial\Omega\times[0,T]\}\cup\{\Omega\times \{0\}\}.
\end{equation*}
For our game, we also need the parabolic boundary strip of width $\varepsilon>0$
\begin{equation*}
    \Gamma_{p}^{\varepsilon}=\left(S_{\varepsilon}\times(-\frac{\varepsilon^{2}}{2},T]   \right)\cup \left(\Omega\times (-\frac{\varepsilon^{2}}{2},0]   \right),
\end{equation*}
where
\begin{equation*}
    S_{\varepsilon}=\{x\in \mathbb{R}^{N}\setminus\Omega,~dist(x,\partial\Omega)\leq \varepsilon\}
\end{equation*}
is the $\varepsilon$-boundary strip of $\Omega$.  

Also, let $F:\Gamma_{p}^{\varepsilon}\longrightarrow\mathbb{R}$ and $\psi:\mathbb{R}^{N+1}\longrightarrow\mathbb{R}$ be two Lipschitz continuous functions such that
\begin{align}
    |F(x,t_x)-F(y,t_y)|\leq C_1 \left( |x-y|+|t_x-t_y|^{\frac{1}{2}} \right),\\
    |\psi(x,t_x)-\psi(y,t_y)|\leq C_2 \left( |x-y|+|t_x-t_y|^{\frac{1}{2}} \right),
\end{align}
and 
\begin{equation}
    \psi\leq F~~~~~~~~~~~~\text{in}~~~~\Gamma^{\varepsilon}_p.
\end{equation}

Next, we define our parabolic obstacle tug of war game with noise. The game is a zero-sum stochastic game between player I and player II in $\Omega_T$. Fix $\varepsilon>0$ and let $\alpha=\frac{p-2}{p+n}$ and $\beta=\frac{n+2}{p+n}$. The game proceeds as follows: First, a token is placed at $(x_0,t_0)\in \Omega_T$. With probability $\alpha$, the players flip a fair coin, and the winner moves the token to $(x_1,t_1)\in B_{\varepsilon}(x_0)\times\{t_0-\frac{\varepsilon^{2}}{2}\}$ according to their strategy, where $B_{\varepsilon}(x_0)$ is an open ball centered at $x_0$ with radius $\varepsilon$. With probability $\beta$, the token moves uniformly at random to $(x_1,t_1)\in B_{\varepsilon}(x_0)\times\{t_0-\frac{\varepsilon^{2}}{2}\}$. From $(x_1,t_1)$, the game continues analogously. Once the token has reached $(x_{\tau^{0}_{N}}, t_{\tau^{0}_{N}})$ ,such that 
\begin{equation}
    \tau^{0}_{N}=\min\{N,\inf\{k,~x_k \in S_\varepsilon,~k=0,1,..,N     \}\},
\end{equation}
the game ends, and player I earns $F(x_{\tau^{0}_{N}},t_{\tau^{0}_{N}})$ while player II pays $F(x_{\tau^{0}_{N}},t_{\tau^{0}_{N}})$. In addition, at every point $(x_k,t_k)$, player I is allowed to end the game earlier and earn $\psi(x_k,t_k)$ while player II pays $\psi(x_k,t_k)$.

Let us denote $H=\Omega\cup S_{\varepsilon}$. A run game is a sequence
\begin{equation*}
    x=(x_0,x_1,...,x_N)\in H^{N+1},
\end{equation*}
where every $x_k$ except $x_0$ is a random variable, depending on the coin tosses, the strategies adopted by the players, and the stopping rule chosen by player I. Notice that $H^{N+1}$, as a product space, has a measurable structure. Then, the final payoff for player I is defined by
\begin{equation}
    G(x,t)=\begin{cases}
        F(x,t)~~~~~\text{if}~~(x,t)\in\Gamma^{\varepsilon}_p,\\
        \psi(x,t)~~~~~\text{if}~~(x,t)\in\Omega_T,
    \end{cases}
\end{equation}
induces a Borel measurable function on $H^{N+1}$.

A strategy $\sigma_{I}$ for player I (resp $\sigma_{II}$ for player II) is a function that, given the history $(x_0,...,x_k)$, selects the next move $\sigma_{I}(x_0,..,x_k)=x_{k+1}\in B_{\varepsilon}(x_{k})$ (similarly for $\sigma_{II}$). The fixed starting point $(x_0,t_0)$, the number of rounds $N$, the domain $\Omega$, the strategies $\sigma_{I}$ and $\sigma_{II}$, and any stopping time $\tau_{N}$ such that
\begin{equation}
\tau_{N}\leq \tau^{0}_{N},
\end{equation}
determine transition probabilities at each step. By the Ionescu--Tulcea theorem, these transition probabilities yield a unique probability measure $\mathbb{P}^{x_0,N}_{\tau_N,\sigma_{I},\sigma_{II}}$ in $H^{N+1}$. This measure is built by using the initial distribution $\delta_{x_{0}}(A)$ and the family of transition probabilities
\begin{equation}
    \pi_{k}(x_0,...,x_k)(A)=\begin{cases}
        \frac{\alpha}{2}\left\{\delta_{\sigma_{I}^{k}(x_0,..,x_k)}(A)+\delta_{\sigma_{II}^{k}(x_0,..,x_k)}(A)\right\}+\beta\frac{|A\cap B_{\varepsilon}(x_k)|}{|B_{\varepsilon}(x_k)|},&~~~~\text{if}~k<\tau_{N},\\
        \delta_{x_k}(A),&~~~~\text{elsewise}.
    \end{cases}
\end{equation}
By using the above setting, each game stops surely since $\tau_{N}\leq N$.

Finally, for any starting point $(x_0,t_0)$, the value of the game for player I with the maximum number of rounds N and (2.6) is given by
\begin{equation}
    u^{\varepsilon}_{I}(x_0,t_0)=\underset{\tau_N,\sigma_{I}}{\sup}\underset{\sigma_{II}}{\inf}~E^{x_0,N}_{\tau_N, \sigma_{I},\sigma_{II}}[G(x_{\tau_N},t_{\tau_N})],
\end{equation}
while the value of the game for player II is given by
\begin{equation}
    u^{\varepsilon}_{II}(x_0,t_0)=\underset{\sigma_{II}}{\inf}\underset{\tau_N,\sigma_{I}}{\sup}~E^{x_0,N}_{\tau_N, \sigma_{I},\sigma_{II}}[G(x_{\tau_N},t_{\tau_N})].
\end{equation}
\section{Dynamic Programming Principle: Existence and Uniqueness}
The Dynamic Programming Principle serves as the cornerstone for characterizing the value function in our parabolic obstacle tug-of-war game. In this section, we establish the existence and uniqueness of solutions to the DPP, which shares structural similarities with the Wald-Bellman equations for optimal stopping problems \cite{pes}.

\begin{thm}
    Let $\alpha\in [0,1]$ and $\beta=1-\alpha$. Let  $F:\Gamma_{p}^{\varepsilon}\longrightarrow\mathbb{R}$ and $\psi:\mathbb{R}^{N+1}\longrightarrow\mathbb{R}$ be two bounded, Borel function satisfying (2.1)-(2.3). Then, there exists a unique bounded Borel function $u^{\varepsilon}:\Omega_T\cup \Gamma_{p}^{\varepsilon}\longrightarrow\mathbb{R}$ satisfying the following DPP
    \begin{equation}\label{eq:DPP}
\begin{aligned}
u^{\varepsilon}(x,t)
&=\max\Big\{\psi(x,t),\;
\frac{\alpha}{2}\Big(\sup_{y\in B_{\varepsilon}(x)}u^{\varepsilon}\big(y,t-\tfrac{\varepsilon^{2}}{2}\big)
+\inf_{y\in B_{\varepsilon}(x)}u^{\varepsilon}\big(y,t-\tfrac{\varepsilon^{2}}{2}\big)\Big)\\
&\qquad\qquad\qquad\qquad\qquad\qquad
+\beta\;\fint_{B_{\varepsilon}(x)}u^{\varepsilon}\big(y,t-\tfrac{\varepsilon^{2}}{2}\big)\,dy\Big\},
\qquad (x,t)\in\Omega_T,\\[6pt]
u^{\varepsilon}(x,t)&=F(x,t),\qquad\qquad\qquad\qquad\qquad\qquad\qquad (x,t)\in\Gamma_{p}^{\varepsilon}.
\end{aligned}
\end{equation}
\end{thm}
\begin{proof}
We define the operator $T$ on bounded Borel function $v:\Omega_T\cup \Gamma_{p}^{\varepsilon}\longrightarrow\mathbb{R}$ by
\begin{equation}
\mathcal{T}v(x,t) =
\begin{cases}
\displaystyle
\max\!\Biggl\{
\psi(x,t),~
\frac{\alpha}{2}\!\left(
\underset{y\in B_{\varepsilon}(x)}{\sup} v\!\left(y,t-\tfrac{\varepsilon^{2}}{2}\right)
+ 
\underset{y\in B_{\varepsilon}(x)}{\inf} v\!\left(y,t-\tfrac{\varepsilon^{2}}{2}\right)
\right)
\\[6pt]
\displaystyle\hspace{3.5cm}
+~\beta \fint_{B_{\varepsilon}(x)} v\!\left(y,t-\tfrac{\varepsilon^{2}}{2}\right)\!dy
\Biggr\}, 
& (x,t)\in\Omega_T, \\[10pt]
F(x,t), & (x,t)\in\Gamma^{\varepsilon}_{p}.
\end{cases}
\end{equation}

We construct a sequence $\{u_{k}\}_{k=0}^{+\infty}$ iteratively
\begin{equation}\begin{cases}
    u_{0}(x,t)=\begin{cases}
        \psi(x,t)~~~&\text{for}~(x,t)\in \Omega_{T},\\
        F(x,t)~~~&\text{for}~(x,t)\in \Gamma_{p}^{\varepsilon},
    \end{cases}\\
    u_{k}(x,t)=\mathcal{T}u_{k}(x,t).
\end{cases}
\end{equation}
We prove by induction that for each $k$, the function $u_k$ is well-defined and Borel measurable, and that the sequence stabilizes after finitely many steps at each point, i.e., for every $k\geq 0$ and $i\geq k$
\begin{equation}
    u_{i}(y,s)=u_{k}(y,s)~~~\text{for all}~y\in \mathbb{R}^{N}~~\text{and}~~s\leq k\frac{\varepsilon^2}{2}.
\end{equation}
Indeed, for $k=0$, $u_0$ by construction is Borel measurable and satisfies (3.4). Now, we will assume that (3.4) holds for some $k>0$ where $u_k$ is Borel measurable, and we will prove that it also holds for $k+1$. Therefore, we fix $s\leq (k+1)\frac{\varepsilon^{2}}{2}$ and let $y\in \mathbb{R}^{N}$. Therefore, by using (3.4) for $k$ and since $s-\frac{\varepsilon^{2}}{2}\leq k\frac{\varepsilon^{2}}{2}$, we have
\begin{equation*}
    u_{k+1}(y,s)=\mathcal{T}u_{k}(y,s)=\mathcal{T}u_{i}(y,s)=u_{i+1}(y,s).
\end{equation*}
Also, since $u_k$ and $\psi$ are Borel measurable and $\sup,~\inf$ and the average preserve Borel measurability, then $u_{k+1}$ is also Borel measurable. Next, since the time $T$ is finite and from (3.4), the following limit
\begin{equation*}
    u^{\varepsilon}(x,t)=\underset{k\rightarrow +\infty}{\lim} u_{k}(x,t)
\end{equation*}
exists and it is reached in finitely many steps at each point. Thus
\begin{equation*}
    u^{\varepsilon}(x,t)=\mathcal{T}u(x,t).
\end{equation*}
For uniqueness, suppose that $v^{\varepsilon}$ is another solution. We show by induction on time that $v^{\varepsilon}=u_k$, for all $k$, hence $u^{\varepsilon}=v^{\varepsilon}$. Indeed, for the case $t\leq0$, both solutions equal $F$ on $\Gamma^{\varepsilon}_{p}$. Assume $v^{\varepsilon}=u_{k}$ for all time up to $t-\frac{\varepsilon^{2}}{2}$. Then 
\begin{equation*}
    v^{\varepsilon}(x,t)=\mathcal{T}v^{\varepsilon}(x,t)=\mathcal{T}u_{k}(x,t)=u_{k+1}(x,t).
\end{equation*}
Hence, by passing to the limit, we get the desired uniqueness.   
\end{proof}
\begin{cor}
Assume that $F_{1} \geq F_{2}$ on $\Gamma_{p}^{\varepsilon}$ and $\psi_{1} \geq \psi_{2}$ in $\Omega_{T}$. 
Let $u_{1}$ and $u_{2}$ denote the corresponding unique solutions to (3.1). 
Then
\begin{equation*}
    u_{1} \geq u_{2} \qquad \text{in } \Omega_{T}.
\end{equation*}
\end{cor}
\begin{proof}
We respect the notation in the proof of Theorem 3.1. The operator $\mathcal{T}$ is monotone by construction. Choose initial iterates $u^{1}_{0},~u^{1}_{0}$ with $u^{1}_{0}\geq u_{0}^{2}$ defined as in (3.3) and define $u^{j}_{i+1}=\mathcal{T} u_{i}^{j}$ for $j=1,2$. By monotonicity we have $u^{1}_{i}\geq u_{i}^{2}$ for all $i$. For each $(x,t)$ the sequence stabilize pointwise to $u^{1}(x,t)$ and $u^{2}(x,t)$, so the inequality pass to the limit, giving $u^{1}\geq u^{2}$.\end{proof}
Let $0 < \varepsilon < \varepsilon_0$ denote the spatial step size. For each $0 < t < T$, define the integer $N(t)$ by
\[
\frac{2t}{\varepsilon^2} \leq N(t) < \frac{2t}{\varepsilon^2} + 1,
\]
that is, we use the shorthand notation \( N(t) = \lfloor \tfrac{2t}{\varepsilon^2} \rfloor \).  
Set \( t_0 = t \) and define the discrete time levels recursively by
\[
t_{k+1} = t_k - \frac{\varepsilon^2}{2}, \quad k = 0, 1, \ldots, N(t) - 1.
\]
Consequently,
\[
t_k  = t_{N(t)} + \frac{\varepsilon^2}{2}\,(N(t) - k), \quad k = 0, 1, \ldots, N(t),
\]
where \( t_{N(t)} \in \big(-\tfrac{\varepsilon^2}{2}, 0\big] \).  
When no confusion arises, we simply write \( N \) for \( N(t) \).  

Next, we establish that the game possesses a well-defined value. Combined with the previous results, this leads to the following theorem.

\begin{thm}
    The value functions of tug of war with noise $u_{I}$ and $u_{II}$ defined in (2.8) and (2.9) respectively, with payoff function $G$ defined in (2.5), coincide with the DPP function $u$ defined (3.1).
\end{thm}
\begin{proof}
  First, we want to prove that 
  \begin{equation}
     u_{II}\leq u~~~\text{in}~\Omega_{T}. 
  \end{equation}
  Fix $(x_0,t_0)\in \Omega_{T}$ so that $N = \lfloor \frac{2t_{x}}{\varepsilon^2} \rfloor$. For any strategy $\sigma_{I}$ and stopping time $\tau_{N}\leq \tau^{0}_{N}$ where $\tau^{0}_{N}$ is defined in (2.4), we construct a strategy $\sigma_{II}^{0}$ for player II such that at $x_{k-1}\in \Omega$ we chooses to step to a point that almost minimizes $u$, this is, to a point $x_k\in B_{\varepsilon}(x_{k-1})$ such that
  \begin{equation*}
      u(u_{k},t_{k})\leq \underset{B_{\varepsilon}(x_{k-1})}{\inf} u(y,t_{k})+\frac{\eta}{2^{k}},
  \end{equation*}
for some fixed $\eta>0$. Moreover, this strategy can be chosen to be a Borel by using Lusin's countable selection theorem \cite{parv2}.

It follows from the choice of the strategies and the DPP (3.1), that
\begin{equation*}
\begin{split}
    E_{\tau_{N},\sigma_{I},\sigma_{II}^{0}}^{x_0,N}&[ u^{\varepsilon}(x_k,t_k)+\frac{\eta}{2^{k}}|x_0,...,x_{k-1} ]\\
    =&\frac{\alpha}{2}\left\{u(\sigma^{k-1}_{I}(x_0,..,x_{k-1}),t_k) +u(\sigma^{k-1}_{II}(x_0,..,x_{k-1}),t_k)    \right\}\\
    &+\beta\fint_{B_{\varepsilon}(x_{k-1})}u(y,t_k)~dy\\
    \leq& \frac{\alpha}{2}\left\{\underset{y\in B_{\varepsilon}(x_{k-1})}{\sup}u(y,t_k)+\underset{y\in B_{\varepsilon}(x_{k-1})}{\inf}u(y,t_k) \right\}\\
    &+\beta\fint_{B_{\varepsilon}(x_{k-1})}u(y,t_k)~dy+\frac{\eta}{2^{k}}(1+\frac{\alpha}{2})\\
\leq& ~u(x_{k-1},t_{k-1})+\frac{\eta}{2^{k-1}}.
\end{split}    
\end{equation*}
Thus, regardless of the strategy $\sigma_{I}$, the process $M_{k}=u(x_k,t_k)+\frac{\eta}{2^{k}}$ is supermartingale with respect to the history of the game. It follows that
\begin{equation}
\begin{split}
    u_{II}(x_0,t_0)=&\underset{\sigma_{II}}{\inf}\underset{\tau_{N},\sigma_{I}}{\sup}E_{\tau_{N},\sigma_{I},\sigma_{II}^{0}}^{x_0,N}[G(x_{\tau_{N}},t_{\tau_{N}})+\frac{\eta}{2^{\tau_{N}}}]\\
    \leq& \underset{\tau_{N},\sigma_{I}}{\sup}~E_{\tau_{N},\sigma_{I},\sigma_{II}^{0}}^{x_0,N}[M_{\tau_{N}}]\leq 
    \underset{\sigma_{I}}{\sup}~E_{\sigma_{I},\sigma_{II}^{0}}^{x_0,N}[M_{0}]=u(x_0,t_{0})+\eta,
\end{split}    
\end{equation}
where we used the fact that $G\leq u$, and the Doob's optional stopping theorem in view of the supermatingae property and the uniform boundedness of $M_{\tau_{N}}$. Since $\eta>0$, then (3.5) follows.

Next, we are going to prove that 
  \begin{equation}
     u\leq u_{I}~~~\text{in}~\Omega_{T}. 
  \end{equation}
Fix $\eta>0$ and fix any strategy $\sigma_{II}$. Player I employs a strategy $\sigma_{I}^{0}$ that maximizes $u$ at each step such that at $x_{k-1}\in\Omega$, he choose to step to a point $x_{k}\in B_{\varepsilon}(x_{k-1})$ that almost maximizes $u$ such that
\begin{equation*}
    u(x_k,t_{k})\geq \underset{B_{\varepsilon}(x_{k-1})}{\sup}~u(y,t_k)-\frac{\eta}{2^k}.
\end{equation*}
Define the following stopping time $\tau^*_{N}$ such that
\[
\tau_N^{*}
= \min\!\Bigl\{
N,\;
\inf\!\bigl\{
k \in \{0,\dots,N\} :
(x_k,t_k) \in \Gamma_{p}^{\varepsilon}
\ \text{or}\
u(x_k,t_k) = \psi(x_k,t_k)
\bigr\}
\Bigr\}.
\]
Note that $\tau^{*}_{N}\leq \tau_{N}\leq N$. Thereafter, consider $k$ such that $k-1<\tau^{*}_{N}$ so that $(x_{k-1},t_{k-1})\notin \Gamma_{p}^{\varepsilon}$ and $u(x_{k-1},t_{k-1})>\psi(x_{k-1},t_{k-1})$. Then, from  (3.1) we get
\begin{equation*}
    u(x_{k-1},t_{k-1})=\frac{\alpha}{2}\left\{\underset{y\in B_{\varepsilon}(x_{k-1})}{\sup} u(y,t_{k})  +\underset{y\in B_{\varepsilon}(x_{k-1})}{\inf} u(y,t_{k})   \right\}+\beta \fint_{B_{\varepsilon}}u(y,t_k)~dy.
\end{equation*}
Therefore,
\[
\begin{split}
E^{x_{0},N}_{\tau^{*}_{N},\sigma_{I}^{0},\sigma_{II}}&[u(x_k,t_k)-\frac{\eta}{2^{k}}|x_0,..,x_k-1]\\
=&\frac{\alpha}{2}\left\{u(\sigma^{k-1}_{I}(x_0,..,x_{k-1}),t_k) +u(\sigma^{k-1}_{II}(x_0,..,x_{k-1}),t_k)    \right\}\\
    &+\beta\fint_{B_{\varepsilon}(x_{k-1})}u(y,t_k)~dy\\
\geq&  \frac{\alpha}{2}\left\{\underset{y\in B_{\varepsilon}(x_{k-1})}{\sup}u(y,t_k)+\underset{y\in B_{\varepsilon}(x_{k-1})}{\inf}u(y,t_k) \right\}\\
    &+\beta\fint_{B_{\varepsilon}(x_{k-1})}u(y,t_k)~dy-\frac{\eta}{2^{k}}(1+\frac{\alpha}{2})\\  
=& u(x_{k-1},t_{k-1})- \frac{\eta}{2^{k}}(1+\frac{\alpha}{2})  \geq  u(x_{k-1},t_{k-1})- \frac{\eta}{2^{k-1}}.
\end{split}
\]
The case where $k-1=\tau^{*}_{N}$ is trivial. Consequently, $M_{k}=u(x_k,t_k)-\frac{\eta}{2^{k}}$ is submartingale with respect to the history of the game. As a result, by Doob's optional stopping theorem we have
\[
\begin{split}
u_{I}(x_0,t_0)&\geq \underset{\sigma_{II}}{\inf}E^{x_{0},N}_{\tau^{*}_{N},\sigma_{I}^{0},\sigma_{II}}[G(x_{\tau^{*}_{N}},t_{\tau^{*}_N})-\frac{\eta}{2^{\tau^{*}_{N}}}]=\underset{\sigma_{II}}{\inf}E^{x_{0},N}_{\tau^{*}_{N},\sigma_{I}^{0},\sigma_{II}}[M_{\tau^{*}_{N}}]\\
&\geq \underset{\sigma_{II}}{\inf}E^{x_{0},N}_{\tau^{*}_{N},\sigma_{I}^{0},\sigma_{II}}[u(x_0,t_0)-\eta]=u(x_0,t_0)-\eta,
\end{split}
\]
where we used the definition of $u_{I}$, the fact that $G(x_{\tau^{*}_{N}},t_{\tau^{*}_{N}})=u(x_{\tau^{*}_{N}},t_{\tau^{*}_{N}})$ derived from the definition of $\tau^{*}_{N}$. Since $\eta>0$ was arbitrary, (3.7) follows. 
\end{proof}
\section{Compactness and convergence analysis}
This section establishes the compactness framework needed to pass to the limit in the discrete game. 
We prove uniform boundedness and a uniform modulus of continuity for the discrete value functions \(u^\varepsilon\), with constants independent of the discretization parameter \(\varepsilon\). 
By Arzelà–Ascoli this yields relative compactness in \(C(\overline{\Omega_T})\), so one may extract a subsequence converging uniformly to a continuous limit. 
Finally, identifying this limit as the unique viscosity solution of (1.1) gives the convergence of the discrete values to the continuous solution. The proof of the following lemma can be found in \cite{man1}.
\begin{lem}
    Let$\{u^{\varepsilon}:\overline{\Omega_{T}}\longrightarrow\mathbb{R},~\varepsilon>0\}$ be a set of functions such that
    \begin{itemize}
        \item there exists $C>0$ so that $|u^{\varepsilon}(x,t)|<C$ for every $\varepsilon>0$ and every $(x,t)\in\overline{\Omega_T}$,
        \item given $\eta>0$, there are constants $r_0$ and $\varepsilon_0$ such that for every $\varepsilon<\varepsilon_0$ and any $(x,t),(y,s)\in \overline{\Omega_T}$ with $|x-y|+|t-s|<r_0$, it holds that
        \[
        |u^{\varepsilon}(x,t)-u^{\varepsilon}(y,s)|<\eta.
        \]
    \end{itemize}
    Then there exists a uniformly continuous function $u:\overline{\Omega_T}\longrightarrow\mathbb{R}$ and a subsequences still denoted by $\{u^{\varepsilon}\}$ such that
    \[
    u^{\varepsilon}\longrightarrow ~u~~~~~\text{uniformly in}~\overline{\Omega_T}
    \]
    as $\varepsilon\longrightarrow 0$.
\end{lem}
Next, we introduce a boundary regularity condition for the domain $\Omega$.
\begin{defn}
    We say that a domain $\Omega$ satisfies an exterior sphere condition if for any $y\in \partial\Omega$, there exists $B_{\delta}(z)\subset \mathbb{R}^{N}\setminus\Omega$ with $\delta>0$ such that $y\in \partial B_{\delta}(z)$.
\end{defn}
Throughout this section, we assume that $\Omega$ satisfies the exterior sphere condition and $\Omega\subset B_{R}(z)$ for some $R>0$. First, we consider the case where $(y,t_y)$ is a point in the lateral boundary strip.
\begin{lem}
   Let  $F:\Gamma_{p}^{\varepsilon}\longrightarrow\mathbb{R}$ and $\psi:\mathbb{R}^{N+1}\longrightarrow\mathbb{R}$ be two bounded, Borel function satisfying (2.1)-(2.3). Let $u^{\varepsilon}:\Omega_T\cap\Gamma_{p}^{\varepsilon}\longrightarrow\mathbb{R}$ be the solution (3.1)  such that
   \begin{equation}
       \begin{split}
           |u^{\varepsilon}(x,t)-u^{\varepsilon}(y,t_y)|\leq &C\bigr(\min\{|x-y|^{\frac{1}{2}}+\varepsilon^{\frac{1}{2}},t_{x}^{\frac{1}{2}}+\varepsilon\}+\min\{|x-y|+\varepsilon,t_{x}+\varepsilon^{2}\}\bigl)\\
           & +C_{F,\psi}\left(|t_{x}-t_{y}|^{\frac{1}{2}}+|x-y|+\delta  \right)
       \end{split}
   \end{equation}
   for every $(x,t_x)\in \Omega_{T}$ and $y\in S_{\varepsilon}$
\end{lem}
\begin{proof}
In the first step, we take $t_x=t_{y}=t_0$, $x=x_0$ and $N = \lfloor \frac{2t_x}{\varepsilon^2} \rfloor$. By the exterior sphere condition, there exists $B_{\delta}(z)\subset \mathbb{R}^{N}\setminus\Omega$ such that $y\in \partial B_{\delta}(z)$. Assume we have fixed a particular strategy $\sigma_{z,II}$ of pulling toward $z$ for player II. Then, by (2.8), we have
\begin{equation}
u^{\varepsilon}(x,t_0)-u^{\varepsilon}(y,t_0)\leq \underset{\tau_{N},\sigma_{I}}{\sup}~E^{x_{0},N}_{\tau_{N},\sigma_{I},\sigma_{z,II}}(G(x_{\tau_{N}},t_{\tau_{N}})-F(y,t_0)).
\end{equation}
Since $\psi(y,t_0)\leq F(y,t_0)$, and by using (2.1)-(2.3), we have
\[
\begin{split}
G(x_{\tau_{N}},t_{\tau_{N}})-F(y,t_0)&\leq C_{F,\psi}(|x_{\tau_{N}}-y|+|t_{\tau_{N}}-t_0|^{\frac{1}{2}})\\
&= C_{F,\psi}\left(|x_{\tau_{N}}-y|+\varepsilon\left( \frac{\tau_{N}}{2}\right)^{\frac{1}{2}}\right).
\end{split}
\]
Then, (4.2) becomes
\begin{equation}
u^{\varepsilon}(x,t_0)-u^{\varepsilon}(y,t_0)\leq C_{F,\psi} \underset{\tau_{N},\sigma_{I}}{\sup}~E^{x_{0},N}_{\tau_{N},\sigma_{I},\sigma_{z,II}}(|x_{\tau_{N}}-y|+\varepsilon\left( \frac{\tau_{N}}{2}\right)^{\frac{1}{2}}).
\end{equation}
By the same method for a fixed strategy $\sigma_{z,I}$ of pulling toward $z$ for player I, and using (2.9) we arrive at
\begin{equation}\begin{split}
    u^{\varepsilon}&(x,t_0)-u^{\varepsilon}(y,t_0)\geq \underset{\sigma_{II}}{\inf}~E^{x_{0},N}_{\tau^{0}_{N},\sigma_{z,I},\sigma_{II}}(G(x_{\tau_{N}^{0}},t_{\tau_{N}^{0}})-F(y,t_0))\\
    =&\underset{\sigma_{II}}{\inf}~E^{x_{0},N}_{\tau^{0}_{N},\sigma_{z,I},\sigma_{II}}(F(x_{\tau_{N}^{0}},t_{\tau_{N}^{0}})-F(y,t_0))
    \geq - C_{F} \underset{\tau_{N},\sigma_{II}}{\sup}~E^{x_{0},N}_{\tau_{N},\sigma_{z,I},\sigma_{II}}(|x_{\tau_{N}}-y|+\varepsilon\left( \frac{\tau_{N}}{2}\right)^{\frac{1}{2}}),
\end{split}
\end{equation}
where the supremum is taken over all admissible stopping times $\tau_{N}\leq \tau_{N}^{0}$. Next, we define the strategy $\sigma_{z,II}$ such that
\begin{equation}
\sigma^{k}_{z,II}(x_0,..,x_k)=\sigma^{k}_{z,II}(x_{k})=\begin{cases}
    x_{k}+(\varepsilon-\varepsilon^{3})\frac{z-x_k}{|z-x_k|}~~~&\text{if}~x_{k}\in \Omega\\
    x_k~~~&\text{if}~x_k\in S_{\varepsilon},
\end{cases}
\end{equation}
and let $\sigma_{I}$ be any strategy for player I. Let $\varepsilon<\frac{\delta}{3}$, then for $(k-1)<\tau_{N}$ we have
\[
\begin{split}
    E^{x_0,N}_{\tau_{N},\sigma_{I},\sigma_{z,II}}&\{|x_{k}-z|-C\varepsilon^{2}k~|x_0,...,x_{k-1}\}\\
    =&\frac{\alpha}{2}|\sigma^{k-1}_{I}(x_0,..,x_k)-z|+\frac{\alpha}{2}|\sigma^{k-1}_{z,II}(x_k)-z|+\beta\fint_{B_{\varepsilon}(x_{k-1})}|\omega-z|~d\omega~-C\varepsilon^{2}k\\
    \leq&\frac{\alpha}{2}\{|x_{k-1}-z|+\varepsilon\}+\frac{\alpha}{2}\{|x_{k-1}-z|-(\varepsilon-\varepsilon^{3})\}+\beta |x_{k-1}-z|\\
    &+\beta C_{\delta}\varepsilon^{2}-C\varepsilon^{2}k\\
    \leq&|x_{k-1}-z|-C\varepsilon^{2}(k-1),
    \end{split}
\]
where we have used the fact that $|x_{k-1}-z|>\varepsilon-\varepsilon^{3}$ and the estimate
\[
\fint_{B_{\varepsilon}(x_{k-1})}|\omega-z|~d\omega\leq |x_{k-1}-z|+C_{\delta}\varepsilon^{2}
\]
in the first inequality. For the case where $(k-1)\geq \tau_{N}$, we have
\[\begin{split}
E^{x_{0},N}_{\tau_{N},\sigma_{I},\sigma_{z,II}}\{|x_{k}-z|-C\varepsilon^{2}k|~x_0,..,x_{k-1}\}&=|x_{k-1}-z|-C\varepsilon^{2}k\\
&\leq |x_{k-1}-z|-C\varepsilon^{2}(n-1).
\end{split}
\]
Therefor, $M_{k}=|x_{k}-z|-C\varepsilon^{2}k$ is supermartingale. By applying Doob's stopping theorem and Jensen's inequality, we get
\begin{equation}
    \begin{split}
E^{x_{0},N}_{\tau_{N},\sigma_{I},\sigma_{z,II}}&[|x_{\tau_{N}}-y|+\varepsilon(\frac{\tau_{N}}{2})^{\frac{1}{2}}]\leq E^{x_{0},N}_{\tau_{N},\sigma_{I},\sigma_{z,II}}[|x_{\tau_{N}}-z|+\varepsilon(\frac{\tau_{N}}{2})^{\frac{1}{2}}]+\delta\\
&=E^{x_{0},N}_{\tau_{N},\sigma_{I},\sigma_{z,II}}[|x_{\tau_{N}}-z|-C\varepsilon^{2}\tau_{N}]+C\varepsilon^{2}E^{x_{0},N}_{\tau_{N},\sigma_{I},\sigma_{z,II}}[\tau_{N}]+\frac{\varepsilon}{\sqrt{2}}E^{x_{0},N}_{\tau_{N},\sigma_{I},\sigma_{z,II}}[\tau_{N}^{\frac{1}{2}}]+\delta\\
&\leq |x_0-z|+C\varepsilon^{2}E^{x_{0},N}_{\tau_{N},\sigma_{I},\sigma_{z,II}}[\tau_N]+C\varepsilon(E^{x_{0},N}_{\tau_{N},\sigma_{I},\sigma_{z,II}}[\tau_{N}])^{\frac{1}{2}}+\delta\\
&\leq |x_0-y|+C\varepsilon^{2}E^{x_{0},N}_{\tau_{N},\sigma_{I},\sigma_{z,II}}[\tau_N]+C\varepsilon(E^{x_{0},N}_{\tau_{N},\sigma_{I},\sigma_{z,II}}[\tau_{N}])^{\frac{1}{2}}+2\delta
\end{split}
\end{equation}
Next, our main goal is to approximate $E^{x_{0},N}_{\tau_{N},\sigma_{I},\sigma_{z,II}}[\tau_{N}]$. For this reason, we consider a new time-independent game board $Y=B_{R}(z)\supset H$ for some $R>0$ with the same initial token position $x_0$. Let $\overline{\sigma_{I}}$ be an extension of the strategy $\sigma_{I}$ such that
\[
\forall(x_0,..,x_k)\in Y^{k+1},~~\overline{\sigma}_{I}^{k}(x_{0},..,x_{k})=\begin{cases}
 \sigma_{I}^{k}(x_{0},..,x_{k})~~~&\text{if}~(x_{0},..,x_{k})\in H^{k+1},\\
 x_{k}~~~&\text{elsewise}.
\end{cases}
\]
As for the strategy for player II, we define
\[
\overline{\sigma}_{z,II}^{k}(x_{0},..,x_{k})=\overline{\sigma}^{k}_{z,II}(x_{k})=\begin{cases}
    x_{k}+(\varepsilon-\varepsilon^{3})\frac{z-x_k}{|z-x_{k}|}~~~&\text{if}~x_{k}\in Y\setminus \overline{B}_{\delta}(z),\\
    x_{k}~~~&\text{elsewise}.
\end{cases}
\]
Let $\overline{\tau}_0:Y^{\infty,x_{0}}\longrightarrow\mathbb{N}\cup\{+\infty\}$ be the exist time into the ball $\overline{B}_{\delta}(z)$, i.e.
\[
\overline{\tau}_0(\omega)=\min\{k\geq0;~|x_{k}-z|\leq \delta\},
\]
and let $\overline{\tau}:Y^{\infty,x_{0}}\longrightarrow\mathbb{N}\cup\{+\infty\}$ be a stopping time extending $\tau_{N}$, so that $\overline{\tau}|_{H^{N+1}}=\tau_{N}$ and $\overline{\tau}\leq\overline{\tau}_0$. Also, $\forall k\geq1$ and $x_{1},..,x_k\in Y$, we define the transition probabilities on $Y$ by
\[
\overline{\pi}_{k}(x_0,...,x_k)(A)=\begin{cases}
    \frac{\alpha}{2}\delta_{\overline{\sigma}_{I}(x_0,..,x_{k})}(A)+\frac{\alpha}{2}\delta_{\overline{\sigma}_{z,II}(x_0,..,x_{k})}(A)+\beta m(x_{k})(A)~~~&\text{for}~x_{k}\in Y\setminus\overline{B}_{\delta}(z),\\
    \alpha \delta_{x_{k}}(A)+\beta m(x_{k})(A)~~~&\text{for}~x_{k}\in B_{\delta}(z)\setminus\overline{B}_{\delta-\varepsilon}(z),\\
    \delta_{x_{k}}(A)~~~&\text{for}~x_{k}\in \overline{B}_{\delta-\varepsilon}(z),
\end{cases}
\]
where the probability $m(x)(A)$ is uniform in the set $B_{\varepsilon}(x)\cap Y$ and is given by:
\[
m(x)(A)=\frac{|A\cap( B_{\varepsilon}(x_k)\cap Y)|}{|B_{\varepsilon}(x_k)\cap Y|}.
\]
Next, we have 
\[
\begin{split}
E^{x_{0},N}_{\tau_{N},\sigma_{I},\sigma_{z,II}}[\tau_{N}]=&\sum_{i=1}^{N}i\mathbb{P}^{x_{0},i}_{\tau_{N},\sigma_{I},\sigma_{z,II}}[\omega\in H^{N+1},~\tau_{N}(\omega)=i]\\
=&\sum_{i=1}^{N}i\mathbb{P}^{x_{0},i}_{\tau_{N},\overline{\sigma}_{I},\overline{\sigma}_{z,II}}[\omega\in H^{N+1},~\tau_{N}(\omega)=i]\\
\leq&\min\{\sum_{i=1}^{\infty}i\mathbb{P}^{x_{0},i}_{\tau_{N},\sigma_{I},\sigma_{z,II}}[\omega\in Y^{\infty,x_{0}},~\overline{\tau}(\omega)=i],N\}=\min\{E^{x_{0}}_{\overline{\sigma}_{I},\overline{\sigma}_{z,II}}[\overline{\tau}],N\}\\
\leq&\min\{E^{x_{0}}_{\overline{\sigma}_{I},\overline{\sigma}_{z,II}}[\overline{\tau}_{0}],N\},
    \end{split}
\]
where we have used the fact that for $k<i=\tau_{N}(\omega)$, we have that $x_{k}\in \Omega$ and 
\[
\overline{\pi}_{k}(x_0,...,x_k)=\pi_{k}(x_0,...,x_k)~~~~~~\forall (x_0,..,x_{k})\in H^{k+1},
\]
and therefore by Lemma 5.2 in \cite{lew}, we have that
\[
\mathbb{P}^{x_{0},i}_{\tau_{N},\sigma_{I},\sigma_{z,II}}[\omega\in H^{N+1},~\tau_{N}(\omega)=i]=\mathbb{P}^{x_{0},i}_{\overline{\sigma}_{I},\overline{\sigma}_{z,II}}[\omega\in H^{N+1},~\tau_{N}(\omega)=i].
\]
Moreover, from the proof of Lemma 5.4 in \cite{lew}, we deduce the following estimate
\begin{equation}
E^{x_{0}}_{\overline{\sigma}_{I},\overline{\sigma}_{z,II}}[\overline{\tau}_{0}]\leq \frac{2C_{\delta}}{\beta\varepsilon^{2}}(|x_0-y|+\varepsilon).
\end{equation}
By putting (4.7) into (4.6), we arrive at
\[\begin{split}
E^{x_{0},N}_{\tau_{N},\sigma_{I},\sigma_{z,II}}[|x_{\tau_{N}}-y|+\varepsilon(\frac{\tau_{N}}{2})^{\frac{1}{2}}]\leq&|x_0-y|+C_{\delta}\min\{|x_{0}-y|+\varepsilon,\varepsilon^{2}N\}\\
&+C_{\delta}\left(\min\{|x_{0}-y|+\varepsilon,\varepsilon^{2}N\}\right)^{\frac{1}{2}}+2\delta
\end{split}
\]
Therefore, (4.2) becomes
\[
\begin{split}
    u^{\varepsilon}(x,t_{0})-u^{\varepsilon}(y,t_{0})&\leq C_{F,\psi,\delta}\min\{|x-y|+\varepsilon,\varepsilon^{2}N\}+C_{F,\psi,\delta}\left(\min\{|x-y|+\varepsilon,\varepsilon^{2}N\}\right)^{\frac{1}{2}}\\
    &+C_{F,\psi}(|x-y|+\delta).
    \end{split}
\]
In the same way, the lower bound can be obtained by choosing for player I, a strategy $\sigma_{z,I}$ by means of (4.5).

Now, for $t_x\neq t_y$ and $y\in S_{\varepsilon}$, we have
\[
\begin{split}
    |u^{\varepsilon}(x,t_{x})-u^{\varepsilon}(y,t_{y})|&\leq  |u^{\varepsilon}(x,t_{x})-u^{\varepsilon}(y,t_{x})|+|u^{\varepsilon}(y,t_{x})-u^{\varepsilon}(y,t_{y})|\\
    &\leq C_{F,\psi,\delta}\min\{|x-y|+\varepsilon,\varepsilon^{2}N\}+C_{F,\psi,\delta}\left(\min\{|x-y|+\varepsilon,\varepsilon^{2}N\}\right)^{\frac{1}{2}}\\
    &+C_{F,\psi}(|x-y|+\delta)+C_{F}|t_{x}-t_{y}|^{\frac{1}{2}}.
\end{split}
\]
Hence, the proof is complete by recalling that $N = \lfloor \frac{2t_x}{\varepsilon^2} \rfloor$.
\end{proof}
Now we consider the case where the boundary point $y,t_y$ lies at the initial boundary strip
\begin{lem}
    Let the assumption of Lemma 4.3 be fulfilled. Then $u^{\varepsilon}$ satisfies the following
    \[|u^{\varepsilon}(x,t)-u^{\varepsilon}(y,t_{y})|\leq C\{|x-y|+t_{x}^{\frac{1}{2}}+\varepsilon\},
    \]
    for every $(x,t_{x})\in \Omega_{T}$ and $(y,t_{y})\in\Omega\times(-\frac{\varepsilon^{2}}{2},0]$.
\end{lem}
\begin{proof}
    Set $x_0=x$ and $N = \lfloor \frac{2t_x}{\varepsilon^2} \rfloor$. We begin by fixing a strategy $\sigma_{y,I}$ for player I to pull toward $y$. Therefore, similar to (4.4), we have
\begin{equation}
    u^{\varepsilon}(x,t_{x})-u^{\varepsilon}(y,t_y)\geq -C_{F}\underset{\tau_{N},\sigma_{II}}{\sup}~E^{x_{0},N}_{\tau_{N},\sigma_{y,I},\sigma_{II}}[|x_{\tau_{N}}-y|+\varepsilon(\frac{\tau_{N}}{2})^{\frac{1}{2}}].
\end{equation}
Next, we have that $M_{k}=|x_{k}-y|^{2}-Ck\varepsilon^{2}$ is supermartingale. Indeed, we have
\[\begin{split}
 E^{x_{0},N}_{\tau_{N},\sigma_{y,I},\sigma_{II}}[|x_{k}-y|^{2}|~x_{0},..,x_{k-1}]\leq&\frac{\alpha}{2}\{(|x_{k}-y|+\varepsilon)^{2}+(|x_{k}-y|-\varepsilon)^{2}\} \\
 &+\beta\fint_{B_{\varepsilon}(x_{k-1})}|x-y|^{2}~dy\leq |x_{k-1}-y|^{2}+C\varepsilon^{2}.
\end{split}
\]
Therefore, by Doob's theorem, we get
\begin{equation}
E^{x_{0},N}_{\tau_{N},\sigma_{y,I},\sigma_{II}}[|x_{\tau_{N}}-y|^{2}]-C\varepsilon^{2}E^{x_{0},N}_{\tau_{N},\sigma_{y,I},\sigma_{II}}[\tau_{N}]\leq |x_{0}-y|^{2}.
\end{equation}
By Jenen's inequality, we get
\[\begin{split}
E^{x_{0},N}_{\tau_{N},\sigma_{y,I},\sigma_{II}}[|x_{\tau_{N}}-y|]\leq&E^{x_{0},N}_{\tau_{N},\sigma_{y,I},\sigma_{II}}[|x_{\tau_{N}}-y|^{2}]^{\frac{1}{2}}\leq |x_{0}-y|+C(t_{x}^{\frac{1}{2}}+\varepsilon),
\end{split}
\]
where we have used the fact that $\tau_{N}\leq N= \lfloor \frac{2t_x}{\varepsilon^2} \rfloor$. Therefore, (4.8) becomes
\begin{equation}
\begin{split}
    u^{\varepsilon}(x,t_{x})-u^{\varepsilon}(y,t_y)&\geq-C_{F}|x-y|-C_{F}(t_{x}^{\frac{1}{2}}+\varepsilon)-C_{F}\varepsilon E^{x_{0},N}_{\tau_{N},\sigma_{y,I},\sigma_{II}}[\tau_{N}^{\frac{1}{2}}]\\
    &\geq-C_{F}\{|x-y|+t_{x}^{\frac{1}{2}}+\varepsilon\}.
\end{split}
\end{equation}
Similarly, by using (4.3) and fixing a strategy $\sigma_{y,II}$ for player II
to pull toward $y$, we get the upper bound
\begin{equation}
    u^{\varepsilon}(x,t_{x})-u^{\varepsilon}(y,t_{y})\leq C_{F,\psi}\{|x-y|+t_{x}^{\frac{1}{2}}+\varepsilon\}.
\end{equation}
Hence, from (4.10) and (4.11), we get the desired result.
\end{proof}
Combining the boundary and interior estimates, we can now verify that the discrete value functions satisfy the uniform continuity and boundedness assumptions required for the compactness argument.
\begin{lem}
     Let the assumption of Lemma 4.3 be fulfilled. Then $u^{\varepsilon}$, the solution to (3.1), satisfies the conditions of Lemma 4.1.
\end{lem}
\begin{proof}
Since $F$ and $\psi$ are Borel bounded, then by construction, we have
\[
|u^{\varepsilon}|\leq M=\max\{\underset{\Omega_{T}}{\sup}\psi,~\underset{\Gamma_{p}^{\varepsilon}}{\sup}F\},
\]
and therefore, $u^{\varepsilon}$ is uniformly bounded.

Next, we are going to handle the second condition of Lemma 4.1 in several cases:\\
For the first case, if $(x,t),~(y,s)\in \Gamma_{p}^{\varepsilon}$, then the values are controlled directly by the boundary data $F$. Then, we can choose $r_0$ small enough such that $|x-y|+|t-s|^{\frac{1}{2}}<r_0$ which implies, by using (2.1) and the fact tha $u^{\varepsilon}=F$ on $\Gamma_{p}^{\varepsilon}$, that
\begin{equation}
    |u^{\varepsilon}(x,t)-u^{\varepsilon}(y,s)|=|F(x,t)-F(y,t)|\leq \eta.
\end{equation}
For the second case, if $(x,t)\in \Omega_{T}$ and $(y,s)\in S_{\varepsilon}\times(-\frac{\varepsilon^{2}}{2},T]$ or $(y,s)\in \Omega\times(-\frac{\varepsilon^{2}}{2},0]$. By taking $\varepsilon_{0},~r_{0}$, and $\delta$ small enough, we can easily deduce from Lemmas 4.3 and 4.4 that
\begin{equation}
    |u^{\varepsilon}(x,t)-u^{\varepsilon}(y,s)|\leq \eta~~~~\text{for}~~|x-y|+|t-s|^{\frac{1}{2}}<r_0.
\end{equation}
 For the last case, if both points are inside  $\Omega_{T}$. We begin by defining the following
\begin{gather*}
d_{p}((z,t),\partial\Omega_{T})
   = \inf\{|z-\omega|^{\frac{1}{2}} + |t-s|^{\frac{1}{2}};\; (\omega,s)\in\partial\Omega_{T}\},\\[4pt]
\widetilde{\Omega_{T}}
   = \{(z,t)\in \Omega_{T}:\; d_{p}((z,t),\partial_{p}\Omega_{T}) > \tfrac{r_{0}}{6}\},\\[4pt]
\widetilde{\Gamma}
   = \{(z,t)\in \overline{\Omega_{T}}:\; d_{p}((z,t),\partial_{p}\Omega_{T}) \leq \tfrac{r_{0}}{3}\}.
\end{gather*}
Next, we assume that $|x-y|+|t_{x}-t_{y}|<\frac{r_0}{3}$. If  $(x,t),~(y,s)\in \widetilde{\Gamma}$, then by the definition of $\widetilde{\Gamma}$ there exist parabolic boundary points $(x_b,t_b)$ and $(y_b,s_b)$ such that
\[
d_{p}((x,t),(x_b,t_b)))\leq \frac{r_0}{3},~\text{and}~~d_{p}((x,t),(y_b,s_b)))\leq \frac{r_0}{3}.
\]
Now, we use the triangle inequality to get
\begin{equation}
    \begin{split}
        |u^{\varepsilon}(x,t)-u^{\varepsilon}(y,s)|\leq& |u^{\varepsilon}(x,t)-u^{\varepsilon}(x_{b},t_{b})|+ |u^{\varepsilon}(x_b,t_b)-u^{\varepsilon}(y_{b},s_{b})|\\
        &+|u^{\varepsilon}(y_b,s_b)-u^{\varepsilon}(y,s)|\leq \eta,
    \end{split}
\end{equation}
where we used the estimates in the previous two cases.\\
Next, If $(x,t_x), (y,t_y) \in \widetilde{\Omega_T}$, we may, without loss of generality, assume that  $t_x > t_y$.  
We then define the bounded Borel function $\tilde{F}: \widetilde{\Omega_T} \longrightarrow \mathbb{R}$ together with the Lipschitz continuous obstacle  $\tilde{\psi}:\mathbb{R}^{N+1}\longrightarrow\mathbb{R}$ as follows.
\begin{align*}
    \tilde{F}(z,t_z)=u^{\varepsilon}(z-x+y,t_z-t_x+t_y)+\eta,\\
    \tilde{\psi}(z,t_z)=\psi^{\varepsilon}(z-x+y,t_z-t_x+t_y)+\eta,
\end{align*}
for all $(z,t_{z})\in\tilde{\Gamma}$. We have by construction that $\tilde{F}\geq \tilde{\psi}$ in $\widetilde{\Omega_{T}}$. Therefore, by Theorem 3.3, there exists a solution $\tilde{u}^{\varepsilon}$ to (3.1) subject to the boundary data $\tilde{F}$ on $\widetilde{\Omega_{T}}$ and the obstacle constraint $\tilde{\psi}$, which by the uniqueness of such solution, is
\[
\tilde{u}^{\varepsilon}=u^{\varepsilon}(z-x+y,t_z-t_x+t_y)+\eta.
\]
Then, similar to (4.14), we have
\[
u^{\varepsilon}(z,t_z)\leq u^{\varepsilon}(z-x+y,t_z-t_x+t_y)+\eta=\tilde{F}(z,t_z).
\]
Similarly, by using (2.2), we get that $\tilde{\psi}\geq \psi$ in $\widetilde{\Omega_{T}}$. Consequently, by using Corollary 3.2, we get that $\tilde{u}^{\varepsilon}\geq u^{\varepsilon}$ in $\widetilde{\Omega_{T}}$ and we get
\[u^{\varepsilon}(x,t_x)\leq \tilde{u}^{\varepsilon}(x,t_x)=u^{\varepsilon}(y,t_y)+\eta.
\]
The lower bound follows similarly.
\end{proof}
The uniform estimates established in the previous lemmas allow us to apply the Ascoli-Arzelà theorem, yielding a subsequence of $u^\varepsilon$ that converges uniformly to a continuous limit $u$ as $\varepsilon \to 0$. We now verify that this limit $u$ is the unique viscosity solution of the parabolic obstacle problem in the sense of the following definition.
\begin{defn}
A continuous function $u:\overline{\Omega_T}\to\mathbb{R}$ is a \emph{viscosity solution} of the obstacle problem
\[
\min\Big\{(n+p)u_t -\big[(p-2)\Delta_\infty u+\Delta u\big],\; u-\psi\Big\}=0\quad\text{in }\Omega_T,
\]
with $u=F$ on the parabolic boundary $\partial_p\Omega_T$ and $u\ge\psi$ in $\Omega_T$, if the following hold.

Let $(x_0,t_0)\in\Omega_T$ and let $\varphi\in C^{2,1}$ be a test function.

\begin{enumerate}[label=(\roman*)]
  \item If $\varphi$ \emph{touches $u$ from below} at $(x_0,t_0)$ (i.e.\ $u-\varphi$ has a local minimum at $(x_0,t_0)$), then
  \begin{equation}
  \begin{cases}
  (n+p)\varphi_t(x_0,t_0)\ge (p-2)\Delta_\infty\varphi(x_0,t_0)+\Delta\varphi(x_0,t_0), 
    & \text{if }\nabla\varphi(x_0,t_0)\neq 0,\\[6pt]
  (n+p)\varphi_t(x_0,t_0)\ge (p-2)\,\lambda_{\min}\!\big(D^2\varphi(x_0,t_0)\big)+\Delta\varphi(x_0,t_0),
    & \text{if }\nabla\varphi(x_0,t_0)=0.
  \end{cases}
   \end{equation}

  \item If $\varphi$ \emph{touches $u$ from above} at $(x_0,t_0)$ (i.e.\ $u-\varphi$ has a local maximum at $(x_0,t_0)$), then either
  \begin{equation}
  u(x_0,t_0)=\psi(x_0,t_0)\quad\text{(contact with the obstacle)},
   \end{equation}
  or the following viscosity inequalities hold:
  \begin{equation}
  \begin{cases}
  (n+p)\varphi_t(x_0,t_0)\le (p-2)\Delta_\infty\varphi(x_0,t_0)+\Delta\varphi(x_0,t_0), 
    & \text{if }\nabla\varphi(x_0,t_0)\neq 0,\\[6pt]
  (n+p)\varphi_t(x_0,t_0)\le (p-2)\,\lambda_{\max}\!\big(D^2\varphi(x_0,t_0)\big)+\Delta\varphi(x_0,t_0),
    & \text{if }\nabla\varphi(x_0,t_0)=0.
  \end{cases}
   \end{equation}
\end{enumerate}

Here $\Delta_\infty\varphi:=\langle D^2\varphi\,\nabla\varphi,\nabla\varphi\rangle/|\nabla\varphi|^2$ (when $\nabla\varphi\neq0$), and $\lambda_{\min}(D^2\varphi)$, $\lambda_{\max}(D^2\varphi)$ denote the smallest and largest eigenvalues of $D^2\varphi$, respectively.  (Equivalently one may write $(p-2)\lambda_{\min}(D^2\varphi)$ and $(p-2)\lambda_{\max}(D^2\varphi)$ in the degenerate case.)
\end{defn}
\noindent This definition specifies the viscosity (obstacle) solution concept for the parabolic normalized $p$-Laplacian obstacle problem. The following convergence theorem constitutes our main result.
\begin{thm}
    Assume that $\Omega$ satisfies the exterior sphere condition and $F$ and $\psi$ satisfy (2.1)-(2.3) respectively. Let $u^{\varepsilon}$ denotes the unique function solving (3.1). Then, the uniform limit $u=\underset{\varepsilon\rightarrow0}{\lim} u^{\varepsilon}$ is a viscosity solution to the parabolic obstacle problem (1.1).
\end{thm}
\begin{proof}
First, we can easily deduce that 
\[
u=F~~~\text{in}~~\partial\Omega_{T},~~~~\text{and}~~~ u\geq\psi~~\text{in}~~\Omega_{T}.
\]
Next, we choose a point $(x,t)\in\Omega_{T},~\varepsilon>0,~s\in(t-\varepsilon^{2},t)$ and any smooth function $\phi\in C^{2,1}$. Also, let $x_{1}^{\varepsilon,t-\frac{\varepsilon^{2}}{2}}$ be a point in which $\phi$ attains its maximum over a ball $\overline{B}_{\varepsilon}(x)$ at time $s$, that is
\[
\phi(x_{1}^{\varepsilon,t-\frac{\varepsilon^{2}}{2}},t-\frac{\varepsilon^{2}}{2})=\underset{y\in \overline{B}_{\varepsilon}(x)}{\max}\phi(y,t-\frac{\varepsilon^{2}}{2}).
\]
Then, from the proof of Theorem 19 in \cite{man3}, we have
\begin{equation}\begin{split}
    \frac{\alpha}{2}&\{\underset{y\in B_{\varepsilon}(x)}{\max}\phi(y,t-\frac{\varepsilon^{2}}{2})+\underset{y\in B_{\varepsilon}(x)}{\min}\phi(y,t-\frac{\varepsilon^{2}}{2})\}+\beta\fint_{B_{\varepsilon}(x)}\phi(y,t-\frac{\varepsilon^{2}}{2})~dy-\phi(x,t)\\
    \leq&\frac{\beta\varepsilon^{2}}{2(n+2)}\biggl( (p-2)<D^{2}\phi(x,t)\frac{x_{1}^{\varepsilon,t-\frac{\varepsilon^{2}}{2}}-x}{\varepsilon},\frac{x_{1}^{\varepsilon,t-\frac{\varepsilon^{2}}{2}}-x}{\varepsilon} >+\Delta\phi(x,t)\\
    &-(n+p)\phi_{t}(x,t)+o(\varepsilon^{2})\biggr).
\end{split}
\end{equation}
Assume further that $\phi$ touches $u$ from above at point $(x,t)$. Then, by the uniform convergence, there exists a sequence $(x_{\varepsilon},t_{\varepsilon})\rightarrow(x,t)$ and a constant $\eta_{\varepsilon}>0$ such that
\[
u^{\varepsilon}(y,s)-\phi(y,s)\leq u^\varepsilon(x_{\varepsilon},t_{\varepsilon})-\phi(x_{\varepsilon},t_{\varepsilon})+\eta_{\varepsilon}
\]
for all $(y,s)$ in the neighborhood of $(x_{\varepsilon},t_{\varepsilon})$. Define the shifted function
\[
\tilde{\phi}(y,s)=\phi(y,s)+[u^{\varepsilon}(x_{\varepsilon},t_{\varepsilon})-\phi(x_{\varepsilon},t_{\varepsilon})],
\]
then
\[
\tilde{\phi}(x_{\varepsilon},t_{\varepsilon})=u^{\varepsilon}(x_{\varepsilon},t_{\varepsilon})~~~\text{and}~~u^{\varepsilon}(y,s)\leq \tilde{\phi}(y,s)+\eta_{\varepsilon}.
\]
Then, $\tilde{\phi}$ serves as an upper test function for $u^{\varepsilon}$ up to a small error. Therefore, by using (3.1), we get
\begin{equation}
\begin{split}
\tilde{\phi}(x_{\varepsilon},t_{\varepsilon})
=& u^{\varepsilon}(x_{\varepsilon},t_{\varepsilon})= \max\Bigg\{
\psi(x_{\varepsilon},t_{\varepsilon}),\,
\frac{\alpha}{2}\Big(
\underset{y\in B_{\varepsilon}(x)}{\sup}\, u^{\varepsilon}\!\left(y,t_{\varepsilon}-\frac{\varepsilon^{2}}{2}\right)
\\
&+\underset{y\in B_{\varepsilon}(x)}{\inf}\, u^{\varepsilon}\!\left(y,t_{\varepsilon}-\frac{\varepsilon^{2}}{2}\right)
\Big)+\beta \fint_{B_{\varepsilon}(x)} u^{\varepsilon}\!\left(y,t_{\varepsilon}-\frac{\varepsilon^{2}}{2}\right) dy
\Bigg\}.
\end{split}
\end{equation}
Then, we have either 
\[
u^\varepsilon(x_{\varepsilon},t_{\varepsilon})=\psi(x_{\varepsilon},t_{\varepsilon}),
\]
which implies by using the uniform convergence and the continuity of $\psi,~\phi$, and $\tilde{\phi}$ that
\[
u(x,t)=\psi(x,t).
\]
Or, we end up with the following inequality
\begin{equation}
    \begin{split}
-\eta_{\varepsilon}\leq&-\tilde{\phi}(x_{\varepsilon},t_{\varepsilon})+\frac{\alpha}{2}\Big(
\underset{y\in B_{\varepsilon}(x)}{\sup}\, \tilde{\phi}\!\left(y,t_{\varepsilon}-\frac{\varepsilon^{2}}{2}\right)
\\
&+\underset{y\in B_{\varepsilon}(x)}{\inf}\, \tilde{\phi}\!\left(y,t_{\varepsilon}-\frac{\varepsilon^{2}}{2}\right)
\Big)+\beta \fint_{B_{\varepsilon}(x)} \tilde{\phi}\!\left(y,t_{\varepsilon}-\frac{\varepsilon^{2}}{2}\right) dy
\Bigg\}.     
    \end{split}
\end{equation}
Next, we have two cases to study. For the first case, we assume that $\nabla\phi(x,t)\neq 0$. Then, by choosing $\eta_{\varepsilon}=o(\varepsilon^{2})$, observing that $\nabla\tilde{\phi}=\nabla\phi,~D^{2}\tilde{\phi}=D^{2}\phi$, and $\tilde{\phi}_{t}=\phi$, and using (4.18), we obtain
\begin{equation}
    \begin{split}
        0\leq& \frac{\beta\varepsilon^{2}}{2(n+2)}\biggl(p-2)<D^{2}\phi(x_{\varepsilon},t_{\varepsilon})\frac{x_{1}^{\varepsilon,t_{\varepsilon}-\frac{\varepsilon^{2}}{2}}-x}{\varepsilon},\frac{x_{1}^{\varepsilon,t_{\varepsilon}-\frac{\varepsilon^{2}}{2}}-x}{\varepsilon}>+\Delta\phi(x_{\varepsilon},t_{\varepsilon})\\
        &-(n+p)\phi_{t}(x_{\varepsilon},t_{\varepsilon})\biggr)+2o(\varepsilon^{2}).
    \end{split}
\end{equation}
Next, we divide (4.21) by $\varepsilon^{2}$ and letting $\varepsilon\rightarrow0$, we get
\[
(n+p)\phi_{t}(x,t)\leq (p-2)\Delta_{\infty}\phi(x,t)+\Delta\phi(x,t).
\]
For the second case, we assume that $\nabla\phi(x,t)=0$. Thereafter, we set $s=t_{\varepsilon}-\frac{\varepsilon^{2}}{2}$, let $y\in B_{\varepsilon}(x_{\varepsilon})$, and we apply Taylor expansion to $\tilde{\phi}$ at $(x_{\varepsilon},t_{\varepsilon})$ such that
\[
\begin{split}
    \tilde{\phi}(y,s)=&\tilde{\phi}(x_{\varepsilon},t_{\varepsilon})+\nabla\phi(x_{\varepsilon},t_{\varepsilon}).(y-x_{\varepsilon})+\frac{1}{2}
    <D^{2}\phi(x_{\varepsilon},t_{\varepsilon})(y-x_{\varepsilon}),(y-x_{\varepsilon})>\\
    &+\phi_{t}(x_{\varepsilon},t_{\varepsilon})(s-t_{\varepsilon})+o(\varepsilon^{2}).
\end{split}
\]
Since $\nabla\phi(x,t)=0$ and $(x_{\varepsilon},t_{\varepsilon})\rightarrow(x,t)$, the linear term contributes only $o(\varepsilon^{2})$ and may be neglected in the limit. Then, the average term becomes
\begin{equation}
\fint_{B_{\varepsilon}(x_{\varepsilon})}\tilde{\phi}(y,s)~dy=\tilde{\phi}(x_{\varepsilon},t_{\varepsilon})+\frac{\varepsilon^{2}}{2(n+2)}\Delta\phi(x_{\varepsilon},t_{\varepsilon})-\frac{\varepsilon^{2}}{2}\phi_{t}(x_{\varepsilon},t_{\varepsilon})+o(\varepsilon^{2}).
\end{equation}
Afterwards, for the external terms, let
\[
M=\underset{y\in B_{\varepsilon}(x_{\varepsilon})}{\sup}~\tilde{\phi}(y,s),~~m=\underset{y\in B_{\varepsilon}(x_{\varepsilon})}{\inf}~\tilde{\phi}(y,s).
\]
Because the gradient vanishes, the extremal points of $\tilde{\phi}$ occur in the directions of the eigenvectors of $D^{2}\phi(x_{\varepsilon},t_{\varepsilon})$. A second-order Taylor expansion gives
\begin{align}
    M=\tilde{\phi}(x_{\varepsilon},t_{\varepsilon})+\frac{\varepsilon^{2}}{2}\lambda_{\max}(D^{2}\phi(x_{\varepsilon},t_{\varepsilon}))-\frac{\varepsilon^{2}}{2}\phi_{t}(x_{\varepsilon},t_{\varepsilon})+o(\varepsilon^{2}),\\
     m=\tilde{\phi}(x_{\varepsilon},t_{\varepsilon})+\frac{\varepsilon^{2}}{2}\lambda_{\min}(D^{2}\phi(x_{\varepsilon},t_{\varepsilon}))-\frac{\varepsilon^{2}}{2}\phi_{t}(x_{\varepsilon},t_{\varepsilon})+o(\varepsilon^{2}).
\end{align}
Adding (4.23) to (4.24) gives
\begin{equation}
M+m=2\tilde{\phi}(x_{\varepsilon},t_{\varepsilon})+\frac{\varepsilon^{2}}{2}\left(\lambda_{\max}(D^{2}\phi(x_{\varepsilon},t_{\varepsilon}))+\lambda_{\min}(D^{2}\phi(x_{\varepsilon},t_{\varepsilon}))\right)-\varepsilon^{2}\phi_{t}(x_{\varepsilon},t_{\varepsilon})+o(\varepsilon^{2}).
\end{equation}
Therefore, by putting (4.22) and (4.25) into (4.20), we arrive at
\begin{equation}
\begin{split}
    -\eta_{\varepsilon}\leq&\varepsilon^{2}\biggl[\frac{\alpha}{4}\left(\lambda_{\max}(D^{2}\phi(x_{\varepsilon},t_{\varepsilon}))+\lambda_{\min}(D^{2}\phi(x_{\varepsilon},t_{\varepsilon}))\right)\\
    &+\frac{\beta}{2(n+2)}\Delta\phi(x_{\varepsilon},t_{\varepsilon})-\frac{1}{2}\phi_{t}(x_{\varepsilon},t_{\varepsilon})\biggr]+o(\varepsilon^{2}).
\end{split}
\end{equation}
Divide by $\varepsilon^{2}$ and pass to the limit $\varepsilon\rightarrow0$, choosing $\eta_{\varepsilon}=o(\varepsilon^{2})$ and $(x_{\varepsilon},t_{\varepsilon})\rightarrow(x,t)$, to obtain
\begin{equation}
    0\leq \alpha\left(\lambda_{\max}(D^{2}\phi(x_{\varepsilon},t_{\varepsilon}))+\lambda_{\min}(D^{2}\phi(x_{\varepsilon},t_{\varepsilon}))\right)+\frac{2\beta}{n+2}\Delta\phi(x,t)-2\phi_{t}(x,t).
\end{equation}
Since $\alpha=\frac{p-2}{p+n}$ and $\beta=\frac{n+2}{p+n}$, (4.27) becomes
\[
(n+p)\phi_{t}(x,t)\leq \frac{p-2}{2}\left(\lambda_{\max}(D^{2}\phi(x_{\varepsilon},t_{\varepsilon}))+\lambda_{\min}(D^{2}\phi(x_{\varepsilon},t_{\varepsilon}))\right)+\Delta\phi(x,t).
\]
Hence, since \(p\geq2\) we have \(p-2\geq0\), and therefore for any symmetric matrix \(A=D^2\varphi\),
\[
\lambda_{\max}\big((p-2)A\big)=(p-2)\lambda_{\max}(A),\qquad
\lambda_{\min}\big((p-2)A\big)=(p-2)\lambda_{\min}(A).
\]

which implies that
\[
\lambda_{\max}((p-2)D^{2}\phi)\geq\frac{p-2}{2}\left(\lambda_{\max}(D^{2}\phi(x_{\varepsilon},t_{\varepsilon}))+\lambda_{\min}(D^{2}\phi(x_{\varepsilon},t_{\varepsilon}))\right),
\]
we have
\[
(n+p)\phi_{t}(x,t)\leq \lambda_{\max}((p-2)D^{2}\phi)+\Delta\phi(x,t).
\]

To verify the other half of Definition 4.6, we derive the reverse inequalities of (4.20) and (4.21) for all $u\geq\psi$ by considering the minimum point of the test function and choosing a test function $\phi$ that touches $u$ from below. The rest of the argument is analogous by  using the fact that
\[
\lambda_{\min}((p-2)D^{2}\phi)\leq \frac{p-2}{2}\left(\lambda_{\max}(D^{2}\phi(x_{\varepsilon},t_{\varepsilon}))+\lambda_{\min}(D^{2}\phi(x_{\varepsilon},t_{\varepsilon}))\right).
\] 
\end{proof}

Having established convergence of the discrete value functions to a continuous limit satisfying the viscosity inequalities, it remains to prove that this limit is unique.
\begin{lem}
    Let $u$ and $\overline{u}$ be two viscosity solutions to (1.1) in the sense of Definition 4.6. Then,
    \begin{equation}
        u=\overline{u}~~~~\text{in  }~\overline{\Omega_{T}}.
    \end{equation}
\end{lem}
\begin{proof}
    We argue by contradiction. Suppose that $u$ and $v$ are two viscosity solutions of the obstacle problem (1.1) in the sense of Definition 4.6, and that their difference attains a positive maximum at some point $(x_0,t_0)\in\overline{\Omega_{T}}$:
    \begin{equation}
        u(x_0,t_0)-v(x_0,t_0)=\underset{\overline{\Omega_{T}}}{\sup}(u-v)>0.
    \end{equation}
    To compare $u$ and $v$, we apply the method of doubling of variables. Then for each $j\in\mathbb{N}$, we define the following penalized function
    \[
    \omega_{j}(x,t,y,s)=u(x,t)-v(y,s)-\phi_{j}(x,t,y,s)-\eta s,~0<\eta<<1,
    \]
    where
    \[
    \phi_{j}(x,t,y,s)=\frac{j^{2}}{4}|x-y|^{4}+j^{2}|t-s|^{2}.
    \]
    Given that $u$ and $v$ are continuous on the closure of the domain and $\phi_{j}$ is coercive, the function $\omega_{j}$ attains its global maximum at some point $(x_j,t_j,y_j,s_j)\in \overline{\Omega_{T}}^{2}$. Moreover, since $(x_0,t_0)$ is a local maximum for $u-v$, we may assume that
    \[
    (x_j,t_j,y_j,s_j)\longrightarrow (x_0,t_0,x_0,t_0)~~~~\text{as}~~j\rightarrow\infty,
    \]
    and consequently
    \[
    \phi_{j}(x_j,t_j,y_j,s_j)\longrightarrow0~~~~\text{as}~~j\rightarrow\infty.
    \]
    Next, for each $j$ we define
    \[
    \varphi_{u}(x,t)=v(y_j,s_j)+\phi_{j}(x,t,y_j,s_j),~~~~~~\varphi_{v}(y,s)=u(x_j,t_j)-\phi_{j}(x_j,t_j,y,s)-\eta s.
    \]
    Since $(x_j,t_j,y_j,s_j)$ is a global maximizer of $\omega_{j}(x,t,y,s)$, we have for $(x,t)$ near $(x_j,t_j)$ and $(y,s)$ near $(y_j,s_j)$ respectively that
    \[
    (u-\varphi_u)(x,t)\leq(u-\varphi_u)(x_j,t_j),~~~~(v-\varphi_v)(y,s)\geq(v-\varphi_v)(y_j,s_j).
    \]
    As a result, $\varphi_u$ is a smooth test function that touches $u$ from above at $(x_j,t_j)$ and $\varphi_{v}$ is a smooth test function that touches $v$ from below at $(y_j,s_j)$. Therefore,  since $\varphi_{u}$ satisfies the condition (ii) stated in Definition 4.6, we have either $u(x_j,t_j)=\psi(x_j,t_j)$ or (4.17).
For the first case, we have
    \[
    u(x_j,t_j)-v(y_j,s_j)=\psi(x_j,t_j)-v(y_j,s_j)\leq\psi(x_j,t_j)-\psi(y_j,s_j),
    \]
which implies, since $u$ and $v$ are continuous at $(x_j,t_j)\longrightarrow(x_0,t_0)$ as $j\rightarrow\infty$, that 
\[
u(x_0,t_0)-v(x_0,t_0)\leq 0,
\]
which contradict (4.28).

Next, when $u(x_j,t_j)>\psi(x_j,t_j)$ and $v\geq\psi$, the proof follows the same steps as in Lemma~6.2 of~\cite{parv} in the sense of Definition~4.6. Therefore, the desired uniqueness follows.
\end{proof}
\section*{Declarations}
This work was supported by the Carnegie Corporation of New York grant (provided through the AIMS Research and Innovation Center)


\begin{thebibliography}{99}

\bibitem{Caffarelli1977}
Caffarelli, L. A.
\newblock The regularity of free boundaries in higher dimensions.
\newblock \emph{Acta Mathematica}, 139:155--184, 1977.

\bibitem{chap}
Chaplain, M. A. and Stuart, A.
\newblock A mathematical model for the diffusion of tumor angiogenesis factor into the surrounding host tissue.
\newblock \emph{Mathematical Medicine and Biology: A Journal of the IMA}, 8(3):191--220, 1991.

\bibitem{DelPezzoRossi2014}
Del Pezzo, L. M. and Rossi, J. D.
\newblock Tug-of-war games and parabolic problems with spatial and time dependence.
\newblock \emph{Differential and Integral Equations}, 27(3--4):269--288, 2014.

\bibitem{ham1}
El Bahja, H., Hauffen, J. C., Jung, P., et al.
\newblock A physics-informed neural network framework for modeling obstacle-related equations.
\newblock \emph{Nonlinear Dynamics}, 113:12533--12544, 2025.

\bibitem{ham2}
El Bahja, H.
\newblock Regularity for obstacle problems to anisotropic parabolic equations.
\newblock \emph{arXiv preprint} arXiv:2410.01132, 2024.

\bibitem{fri}
Friedman, A.
\newblock Free boundary problems in science and technology.
\newblock \emph{Notices of the AMS}, 47(8):854--861, 2000.

\bibitem{Han2022}
Han, J.
\newblock Time-dependent tug-of-war games and normalized parabolic {$p$}-Laplace equations.
\newblock \emph{Nonlinear Analysis}, 214:112542, 2022.

\bibitem{Kinderlehrer1980}
Kinderlehrer, D. and Stampacchia, G.
\newblock \emph{An Introduction to Variational Inequalities and Their Applications}.
\newblock Academic Press, New York, 1980.

\bibitem{lew}
Lewicka, M. and Manfredi, J. J.
\newblock The obstacle problem for the $p$-Laplacian via optimal stopping of tug-of-war games.
\newblock \emph{Probability Theory and Related Fields}, 167(1--2):349--378, 2017.

\bibitem{man1}
Manfredi, J. J., Parviainen, M., and Rossi, J. D.
\newblock On the definition and properties of $p$-harmonious functions.
\newblock \emph{Annali della Scuola Normale Superiore di Pisa, Classe di Scienze}, 11(2):215--241, 2012.

\bibitem{man2}
Manfredi, J. J., Parviainen, M., and Rossi, J. D.
\newblock An asymptotic mean value characterization for $p$-harmonic functions.
\newblock \emph{Proceedings of the American Mathematical Society}, 138(3):881--889, 2010.

\bibitem{ManfrediParviainen2012}
Manfredi, J. J., Parviainen, M., and Rossi, J. D.
\newblock Dynamic programming principle for tug-of-war games with noise.
\newblock \emph{Control, Optimisation and Calculus of Variations (COCV)}, 18(1):81--90, 2012.

\bibitem{man3}
Manfredi, J. J., Parviainen, M., and Rossi, J. D.
\newblock An asymptotic mean value characterization for a class of nonlinear parabolic equations related to tug-of-war games.
\newblock \emph{SIAM Journal on Mathematical Analysis}, 42(5):2058--2081, 2010.

\bibitem{ManfrediParviainen2013}
Manfredi, J. J., Parviainen, M., and Rossi, J. D.
\newblock An obstacle problem for tug-of-war games.
\newblock \emph{Communications on Pure and Applied Analysis}, 14(1):217--228, 2015.

\bibitem{man4}
Manfredi, J. J., Rossi, J. D., and Somersille, S.
\newblock An obstacle problem for tug-of-war games.
\newblock \emph{Communications on Pure and Applied Analysis}, 14(1):217--228, 2015.

\bibitem{parv}
Parviainen, M. and Ruosteenoja, E.
\newblock Local regularity for time-dependent tug-of-war games with varying probabilities.
\newblock \emph{Journal of Differential Equations}, 261(2):1357--1398, 2016.

\bibitem{parv2}
Parviainen, M.
\newblock \emph{Notes on Tug-of-War Games and the $p$-Laplace Equation}.
\newblock SpringerBriefs on PDEs and Data Science, Springer, Singapore, 2024.

\bibitem{PeresSheffield2008}
Peres, Y. and Sheffield, S.
\newblock Tug-of-war with noise: a game-theoretic view of the $p$-Laplacian.
\newblock \emph{Duke Mathematical Journal}, 145(1):91--120, 2008.

\bibitem{PeresSchramm2009}
Peres, Y., Schramm, O., Sheffield, S., and Wilson, D. B.
\newblock Tug-of-war and the infinity Laplacian.
\newblock \emph{Journal of the American Mathematical Society}, 22(1):167--210, 2009.

\bibitem{pes}
Peskir, G. and Shiryaev, A.
\newblock \emph{Optimal Stopping and Free-Boundary Problems}.
\newblock Lectures in Mathematics, ETH Zürich, Birkhäuser, 2006.

\bibitem{pha}
Pham, H.
\newblock Optimal stopping, free boundary, and American option in a jump-diffusion model.
\newblock \emph{Applied Mathematics and Optimization}, 35:145--164, 1997.

\end{thebibliography}
\end{document}